\documentclass[a4paper,12pt,reqno]{amsart}
\usepackage{amsmath,amsfonts,amssymb,amsthm}
\usepackage{enumerate,enumitem,multicol}
\usepackage{graphicx,epstopdf} 
\usepackage{subcaption}  
\usepackage{tikz}
\usepackage{wrapfig, float}
\usepackage[colorlinks=true, linkcolor=blue, citecolor=blue, urlcolor=blue]{hyperref}
\usepackage{cleveref}
\usepackage[font=footnotesize, labelfont=bf, labelsep=colon]{caption}
\usepackage{cite}
\usepackage{setspace}
\usepackage[a4paper,margin=2.5cm]{geometry}
\newtheorem{theorem}{Theorem}[section]
\newtheorem{proposition}[theorem]{Proposition}

\newtheorem{definition}[theorem]{Definition}
\newtheorem{corollary}[theorem]{Corollary}

\newtheorem{lemma}[theorem]{Lemma}

\newtheorem{remark}[theorem]{Remark}

\allowdisplaybreaks
\DeclareMathOperator{\sgn}{sgn}
\usepackage{mathtools}
\usetikzlibrary{arrows.meta}
\usepackage{autobreak}
\allowdisplaybreaks
\numberwithin{equation}{section}
\numberwithin{figure}{section}
\begin{document}
\begin{center}
\Large{\textbf{Global Solutions to the Discrete Nonlinear Breakage Equations without Mass Transfer}}
\end{center}
\medskip
\medskip
\centerline{${\text{Mashkoor~ Ali$^1$}}$ and ${\text{Philippe~ Lauren\c{c}ot$^2$}}$}\let\thefootnote\relax

\footnotetext{$^1$ E-mail: mashkoor.ali@jgu.edu.in \quad $^{2}$ E-mail: philippe.laurencot@univ-smb.fr}
\medskip
{\footnotesize

  \centerline{ ${}^{1}$ Jindal Global Business School, O.P. Jindal Global University,}
   \centerline{Sonipat-131001, Haryana, India}

\centerline{ ${}^{2}$ Laboratoire de Math\'ematiques (LAMA) UMR 5127, Universit\'e Savoie Mont Blanc, CNRS,}
   \centerline{ F-73000, Chamb\'ery, France}
 
}

\bigskip

\begin{quote}
{\small {\em \bf Abstract.} Global existence of mild solutions to the discrete collisional breakage equations is established for a broad class of collision kernels, without imposing any growth assumptions. In addition, classical solutions are constructed, and uniqueness is proved for an appropriate class of kinetic coefficients and initial data. The large time behavior of solutions is also discussed, and numerical simulations are presented to support the theoretical results.
 }
\end{quote}

\vspace{.3cm}

\section{\textbf{Introduction}} \label{SEC1}

Coagulation-fragmentation processes naturally occur in the dynamics of cluster growth and describe the way a system of clusters can merge to form larger ones or fragment to form smaller ones. Models of cluster growth arise in a wide variety of situations, including aerosol science, astrophysics, colloidal chemistry, polymer science, and biology. In the model considered in this paper, clusters are assumed to be identified by a single parameter, their size, which ranges in the set of positive integers $\mathbb{N}\setminus\{0\}$. Equivalently, each cluster is made of a finite number of identical elementary units and this number is usually referred to as their size. In the following, we shall refer to $i$-clusters for clusters made of $i$ elementary units, $i\ge 1$. In contrast, the size of clusters may take any value in $(0,\infty)$ in the so-called continuous model.

On the one hand, coagulation is inherently nonlinear, as two or more clusters merge to form a larger cluster (see \Cref{Coag}). On the other hand, fragmentation or breakage can be classified into two categories: linear (spontaneous) fragmentation and nonlinear (collision-induced) fragmentation. In the former process, a cluster breaks apart, either spontaneously due to intrinsic instabilities, or through external perturbations, such as mechanical stress or radiation (see \Cref{Frag1}). In nonlinear or collision-induced fragmentation, the collision of two clusters may lead to an exchange of mass between the clusters besides their splitting. A typical example of a collision-induced fragmentation event with mass transfer is the formation of clusters with respective sizes $i+k$ and $j-k$ after the collision of two clusters with respective sizes $i$ and $j>k$, see \Cref{Frag2}. As a consequence, the maximal size of the clusters may increase when mass transfer is possible. Such a phenomenon does not take place in nonlinear fragmentation without mass transfer, which somewhat corresponds to the situation where one of the incoming clusters splits into smaller fragments while the other remains intact, see \Cref{Frag3}.

\begin{figure}[h]
    \centering
    \begin{tikzpicture}
        \draw[fill=gray!20] (-2,0) circle (0.5);
        \node at (-2,0) {$i$};
        
        \node at (-1,0) {$+$};
        
        \draw[fill=gray!20] (0,0) circle (0.5);
        \node at (0,0) {$j$};
        
        \draw[-Stealth, thick] (1,0) -- (2,0);
        
        \draw[fill=gray!50] (3.2,0) circle (0.7);
        \node at (3.2,0) {$i+j$};
    \end{tikzpicture}
    \caption{Illustration of the coagulation process where a $i$-cluster and a $j$-cluster combine to form a $i+j$-cluster.}
    \label{Coag}
\end{figure}
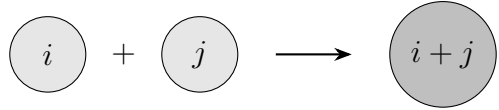

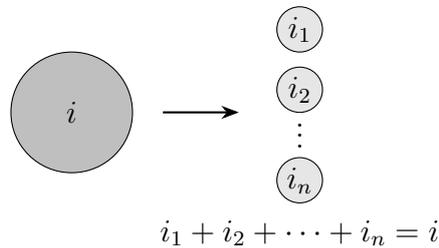
\begin{figure}[h]
\centering
\begin{tikzpicture}
    \draw[fill=gray!50] (-2,0) circle (0.8);
    \node at (-2,0) {$i$};
    
    \draw[-Stealth, thick] (-0.8,0) -- (0.2,0);
    
    \draw[fill=gray!20] (1,1.1) circle (0.3);
    \node at (1,1.1) {$i_1$};
    
    \draw[fill=gray!20] (1,0.3) circle (0.3);
    \node at (1,0.3) {$i_2$};
    
    \node at (1,-0.2) {$\vdots$};
    
    \draw[fill=gray!20] (1,-0.9) circle (0.3);
    \node at (1,-0.9) {$i_n$};
    
    \node at (1,-1.6) {$i_1 + i_2 + \dots + i_n = i$};
\end{tikzpicture}
\caption{Illustration of the fragmentation process without loss of matter, where a $i$-cluster breaks into smaller clusters with respective sizes $i_1, i_2, \dots, i_n$, with the sum of their sizes being equal to that of the original particle.}
\label{Frag1}
\end{figure}

\begin{figure}[h]
\centering
\begin{tikzpicture}
    \draw[fill=gray!20] (-3.2,0.5) circle (0.8);
    \node at (-3.2,0.5) {$i$};
    
    \node at (-2,0.5) {$+$};
    
    \draw[fill=gray!20] (-0.8,0.5) circle (0.8);
    \node at (-0.8,0.5) {$j$};
    
    \draw[-Stealth, thick] (0.3,0.5) -- (1.3,0.5);
    
    \draw[fill=gray!50] (2.5,0.8) circle (0.85);
    \node at (2.5,0.8) {$i + k$};
    
    \node at (3.65,0.5) {$+$};
    
    \draw[fill=gray!30] (4.5,0.2) circle (0.6);
    \node at (4.5,0.2) {$j - k$};
    
\end{tikzpicture}
\caption{Illustration of nonlinear fragmentation process with mass transfer. During the collision, $k$ $1$-clusters are transferred from the incoming $j$-cluster (with $j>k$) to the incoming $i$-cluster, resulting in clusters with respective sizes $i+k$ and $j-k$.}
\label{Frag2}
\end{figure}
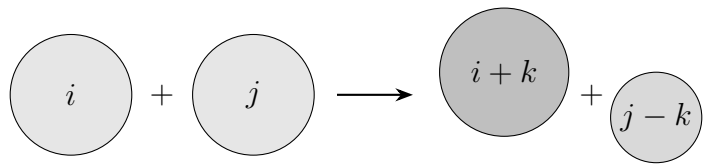

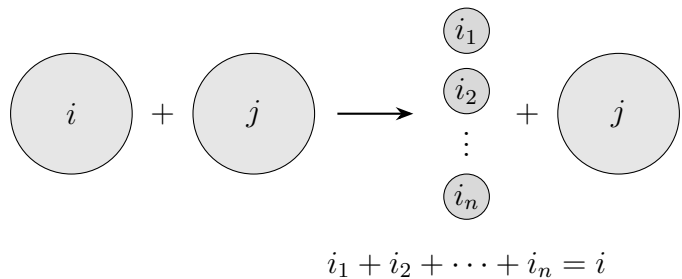
\begin{figure}[h]
    \centering
    \begin{tikzpicture}
        \draw[fill=gray!20] (-3.2,0.5) circle (0.8);
        \node at (-3.2,0.5) {$i$};
        
        \node at (-2,0.5) {$+$};
        
        \draw[fill=gray!20] (-0.8,0.5) circle (0.8);
        \node at (-0.8,0.5) {$j$};
        
        \draw[-Stealth, thick] (0.3,0.5) -- (1.3,0.5);
        
        \draw[fill=gray!30] (2,1.6) circle (0.3);
        \node at (2,1.6) {$i_1$};
        
        \draw[fill=gray!30] (2,0.8) circle (0.3);
        \node at (2,0.8) {$i_2$};
        
        \node at (2,0.2) {$\vdots$};
        
        \draw[fill=gray!30] (2,-0.6) circle (0.3);
        \node at (2,-0.6) {$i_n$};
        
        \node at (2.8,0.5) {$+$};
        
        \draw[fill=gray!20] (4,0.5) circle (0.8);
        \node at (4,0.5) {$j$};
        
        \node at (2,-1.5) {$i_1 + i_2 + \dots + i_n = i$};
        
    \end{tikzpicture}
    \caption{Illustration of nonlinear fragmentation process without mass transfer and without loss of matter. During the collision, a $i$-cluster splits into smaller clusters of sizes $i_1, i_2, \dots, i_n$ such that $i_1 + i_2 + \dots + i_n = i$, while the $j$-cluster remains unchanged.}
    \label{Frag3}
\end{figure}

A widely used approach in the modeling of these processes is based on rate equations, which track the time evolution of cluster size distributions. The first equation of this kind, modeling the coagulation phenomenon, was introduced by the Polish physicist M.~Smoluchowski in his seminal papers \cite{SMOL 1916, SMOL 1917}. The coagulation equation, both with and without linear fragmentation, has been extensively studied over the past few decades, the size variable being either discrete or continuous; for a detailed and thorough review, see \cite{BLL 2019} and the references therein.

In \cite{PhL 2001}, Lauren\c{c}ot and Wrzosek studied the discrete coagulation equation with nonlinear breakage, marking it as the first mathematical study addressing nonlinear breakage. More precisely, denoting by $\psi_i(t)$, $i \ge 1$, the number density of $i$-clusters at time $t \geq 0$, the discrete coagulation equation with nonlinear breakage reads
 \begin{subequations} \label{FNLDCCBE}
\begin{align}
\frac{d\psi_i}{dt}  &=\frac{1}{2} \sum_{j=1}^{i-1} p_{j,i-j} \Gamma_{j,i-j} \psi_j \psi_{i-j}-\sum_{j=1}^{\infty} \Gamma_{i,j} \psi_i \psi_j\nonumber\\
&\quad + \frac{1}{2} \sum_{j=i+1}^{\infty} \sum_{k=1}^{j-1}(1-p_{j-k,k}) \Phi_{j-k,k}^i \Gamma_{j-k,k} \psi_{j-k} \psi_k ,  \hspace{.5cm} i \ge 1, \label{NLDCCBE}\\
\psi_i(0) &= \psi_i^{\rm{in}}, \hspace{.5cm} i \ge 1. \label{NLDCCBEIC}
\end{align}
\end{subequations}
Here, $\Gamma_{i,j}$ denotes the rate of collisions between $i$-clusters and $j$-clusters, while $p_{i,j}$ represents the probability that two colliding clusters with respective sizes $i$ and $j$ merge into a single $i+j$-cluster. The complementary probability, $(1 - p_{i,j})$, corresponds to cluster fragmentation, possibly involving a transfer of matter. The coefficients $(\Gamma_{i,j})$ and $(p_{i,j})$ satisfy the following symmetry property
\begin{equation*}
	0\le p_{i,j} = p_{j,i} \le 1, \qquad \Gamma_{i,j}=\Gamma_{j,i} \geq 0,\qquad i,j\ge 1, 
\end{equation*}
 while  $\{\Phi_{i,j}^s, s=1,2,...,i+j-1\}$  is the size distribution function of the fragments resulting from the collision between a $i$-cluster and a $j$-cluster and satisfies 
\begin{equation}
	\begin{split}
		\Phi_{i,j}^s = \Phi_{j,i}^s & \geq 0, \quad 1\le s \le i+j-1,  \\ 
		\sum_{s=1}^{i+j-1} s \Phi_{i,j}^s & = i+j,
	\end{split} \hspace{1cm} i,j\ge 1. \label{LMC}
\end{equation}
The second identity in~\eqref{LMC} ensures mass conservation during each collisional breakage event, so that conservation of matter is expected throughout time evolution. In terms of the number densities $(\psi_i)_{i\ge 1}$, mass conservation reads
\begin{equation}
	\sum_{i=1}^{\infty} i\psi_i(t) = \sum_{i=1}^{\infty} i\psi_i^{\rm{in}}, \qquad t\ge 0. \label{MCC}
\end{equation}
The first term in~\eqref{NLDCCBE} accounts for the formation of $i$-clusters through coagulation, with a rate determined by the effective coagulation kernel $(p_{i,j}\Gamma_{i,j})$, whereas the second term represents the depletion of $i$-mers due to collisions with other clusters in the system. Finally, the third term describes the creation of $i$-clusters resulting from the collision and subsequent breakup of larger clusters. It is worth noting that the assumption~\eqref{LMC} allows for mass transfer between the colliding clusters, meaning that some of the resulting particles may be larger than either of the incoming particles. In other words, mass transfer between the colliding clusters may occur and the mean size of the system of clusters does not necessarily decrease during the time evolution. Let us also mention here that, when $p_{i,j} = 1$, the equation~\eqref{FNLDCCBE} reduces to the classical Smoluchowski coagulation equation.  In \cite{PhL 2001}, the authors investigate the existence, uniqueness, mass conservation, and long term behavior of weak solutions to~\eqref{FNLDCCBE} under reasonable assumptions on the collision kernel and the daughter distribution function. The study performed in \cite{PhL 2001} also explores the occurrence of the gelation phenomenon; that is, the breakdown of the identity~\eqref{MCC} in finite time. In \cite{AG 24}, the previous work is extended to investigate classical solutions and explore various additional properties. The continuous counterpart of~\eqref{FNLDCCBE} has undergone significant study in recent years; see \cite{PKB 2020, PKB 2020I, AKG 2021}.

In the absence of coagulation ($p_{i,j}=0$), the equation~\eqref{FNLDCCBE} becomes the discrete nonlinear fragmentation equation and reads 
\begin{subequations} \label{FNLDCBE}
\begin{align}
	\frac{d\psi_i}{dt}  &=\frac{1}{2} \sum_{j=i+1}^{\infty} \sum_{k=1}^{j-1} \Phi_{j-k,k}^i \Gamma_{j-k,k} \psi_{j-k} \psi_k -\sum_{j=1}^{\infty} \Gamma_{i,j} \psi_i \psi_j,  \qquad i \ge 1, \label{NLDCBE}\\
	\psi_i(0) &= \psi_i^{\rm{in}}, \qquad i \ge 1. \label{NLDCBEIC}
\end{align}
\end{subequations}
In \cite{AGP 24}, the well-posedness of~\eqref{FNLDCBE} is studied for a broad class of collision kernels and daughter distribution functions which does not exclude mass transfer. In addition, non-trivial stationary solutions are constructed, a result which is closely related to mass exchange during collisions. The continuous version of~\eqref{FNLDCBE} is investigated in \cite{JG 2025}.

A condition on the daughter distribution function $(\Phi_{i,j}^s)$ excluding mass tranfer is provided in \cite{CHNG 90} and reads  
\begin{equation*}
\Phi_{i,j}^s = \textbf{1}_{[s, \infty)} (i) \varphi_{s,i;j} + \textbf{1}_{[s, \infty)} (j) \varphi_{s,j;i} 
\end{equation*}
for $i,j\ge 1$ and $1\le s\le i+j-1$, where $\textbf{1}_{[s, \infty)}$ denotes the characteristic function of the interval $[s,\infty)$. In that case, the system~\eqref{FNLDCBE} reduces to
\begin{subequations} \label{FSNLBE}
\begin{align}
	\frac{d\psi_i}{dt} &= \sum_{j=i+1}^{\infty} \sum_{k=1}^{\infty}   \varphi_{i,j;k}\Gamma_{j,k} \psi_j \psi_k - (1-\delta_{i,1}) \sum_{j=1}^{\infty} \Gamma_{i,j} \psi_i \psi_j, \qquad i\ge 1, \label{SNLBE}\\
	\psi_i(0) &=\psi_{i}^{\rm{in}}, \hspace{.5cm} i\ge 1,\label{SNLBEIC}
\end{align}
\end{subequations}
where $\delta_{1,1}=1$, $\delta_{i,1}=0$, $i\ge 2$, and $\{\varphi_{i,j;k}, 1\leq i \leq j-1\}$ denotes the distribution function of the fragments resulting from the collision of a $j$-cluster with a $k$-cluster and satisfies the conservation of matter
\begin{equation}
	\sum_{i=1}^{j-1} i \varphi_{i,j;k} = j, \hspace{.5cm} j\geq 2,~~~ k\geq 1. \label{LMC1}
\end{equation}
 Since each cluster fragments into smaller clusters after a collision, see Figure~\ref{Frag3}, it is expected that, in the large time, only $1$-clusters will remain.
 
Considerable attention has been given to the study of the continuous version of~\eqref{FSNLBE}, focusing on its analytical solutions. In \cite{CHNG 90}, Cheng and Redner investigate the asymptotic behavior of various classes of models, demonstrating that certain models can be transformed into the linear fragmentation equation through a suitable rescaling of time. This transformation is extensively applied in \cite{EP 2007} to examine the nonlinear fragmentation equation with product collision kernels, addressing the existence and non-existence of solutions, as well as the emergence of finite time singularities. Further insights into the dynamics of the models analyzed in \cite{CHNG 90} are provided by Krapivsky and Ben-Naim in \cite{Krapivsky 2003}. Additionally, the nonlinear fragmentation equation with both product and sum collision kernels is studied by Kostoglou and Karabelas in \cite{Kostoglou 2000}, employing a combination of analytical solutions and asymptotic expansions; see also \cite{Kostoglou 2006}. From a mathematical perspective, the continuous version has been examined in \cite{AKG 2021I, AKG 2024}, with a focus on the analysis of existence, non-existence, and the occurrence of the shattering phenomenon.

As for the discrete setting, the existence of classical solutions to~\eqref{FSNLBE} is established in \cite{AGP 24I} for collision kernels having at most quadratic growth
\begin{equation}
	\Gamma_{i,j} \le A i j, \qquad i,j\ge 1, \label{QG}
\end{equation}
 and a broad class of fragment distribution functions. Various other properties, including uniqueness, differentiability, moment propagation, and large time behavior, are also investigated. Motivated by the recent study \cite{PhL 25}, where the Redner–ben-Avraham–Kahng cluster system, also known as the cluster eating system, is investigated without imposing growth conditions on the kinetic coefficients, the aim of this work is to show that the growth condition~\eqref{QG} can be relaxed and that global mild solutions to~\eqref{FSNLBE} can be constructed assuming only the collision kernel $(\Gamma_{i,j})$ to be non-negative and symmetric; that is,
 \begin{equation}  
 	0\le \Gamma_{i,j} =\Gamma_{j,i}, \qquad i,j\ge 1. \label{ACond}
 \end{equation}
 However, this sole assumption does not seem to be sufficient to derive an existence result for~\eqref{FSNLBE} and, as in \cite{AGP 24I}, it requires to be supplemented with an additional assumption on the fragment distribution function, besides~\eqref{LMC1}: there exist non-negative constants $\alpha_0$ and $\alpha_1$ such that
\begin{equation}
	\varphi_{i,j;k} \leq \alpha_0 + \alpha_1 \varphi_{i,k;j} \quad \text{for all } 1 \leq i \leq j-1, \,\, k \geq j. \label{bCond}
\end{equation}

\begin{remark}\label{RMK2}
	It is important to note that the assumption~\eqref{bCond} holds for a wide range of fragment distribution functions. Indeed, typical examples include bounded fragment distribution functions such as 
	\begin{equation}
		\varphi_{i,j;k} = \frac{i^\nu j}{\displaystyle{\sum_{l=1}^{j-1} l^{1+\nu}}}, \qquad 1 \le i \le j-1, \quad j \ge 2, \quad k \ge 1, \label{FDF:EX1}
	\end{equation}
	for $\nu\ge -1$, which reduces to the uniform distribution 
	\begin{equation*}
		\varphi_{i,j;k} = \frac{2}{j - 1}, \qquad 1 \le i \le j-1, \quad j \ge 2, \quad k \ge 1,
	\end{equation*}
	for $\nu=0$. It also includes unbounded fragment distribution functions like
	\begin{equation*}
		\varphi_{i,j;k} = j \delta_{i,1}, \qquad 1 \le i \le j-1, \quad j \ge 2, \quad k \ge 1,
	\end{equation*}
	which satisfies~\eqref{bCond} with $(\alpha_0,\alpha_1)=(0,1)$, the fragment distribution function~\eqref{FDF:EX1} for $\nu<-1$, which satisfies~\eqref{bCond} with 
	\begin{equation*}
		\alpha_0 = 0, \quad \alpha_1 = \left\{ \begin{array}{ll}
			2^{2+\nu}, & \nu\in [-2,-1), \\
			\displaystyle{\sum_{l=1}^\infty l^{1+\nu}}, & \nu<-2,
		\end{array} \right.
	\end{equation*} 
	and
	\begin{equation}
		\varphi_{i,j;k} = \frac{1}{2^i} \frac{j 2^{j-1}}{2^j - j - 1}, \qquad 1 \le i \le j-1, \quad j \ge 2, \quad k \ge 1, \label{FDF:EX2}
	\end{equation}
	which satisfies~\eqref{bCond} with $\alpha_0=\alpha_1=1$, see \Cref{LEM:EX2} below.
\end{remark}
 
As we shall see below, assumptions~\eqref{LMC1}, \eqref{ACond}, and~\eqref{bCond} allow us to construct mild solutions to~\eqref{FSNLBE} for a large class of initial conditions, including all non-negative sequences having a finite superlinear moment. To be more precise, for $\sigma\ge 0$, let us define the Banach space
\begin{equation*}
	Y_{\sigma} := \left\{ \psi =(\psi_i)_{i\ge 1} :~~~~~ \psi_i \in \mathbb{R}, \ \sum_{i=1}^{\infty} i^{\sigma} |\psi_i| < \infty \right\}
\end{equation*}
with the norm  
\begin{equation*}
	\|\psi\|_{\sigma} := \sum_{i=1}^{\infty} i^{\sigma} |\psi_i|, \qquad \psi\in Y_\sigma,
\end{equation*}
along with its positive cone 
\begin{equation*}
	Y_{\sigma, +} := \left\{ \psi \in Y_{\sigma} : \psi_i \geq 0 \text{ for each } i \ge 1 \right\}.
\end{equation*}
In particular, for a non-negative cluster distribution $\psi$, the norm $\|\psi\|_0$ represents the total number of clusters, while the norm $\|\psi\|_1$ accounts for the total mass of the clusters, so that the conservation of mass~\eqref{MCC} is equivalent to $\|\psi(t)\|_1=\|\psi^{\rm{in}}\|_1$ for all $t\ge 0$. We next introduce the set $\mathcal{G}_1$ of non-negative and convex functions $G \in C^2([0,\infty))$ such that $G(0)=G'(0)=0$ and $G'$ is a concave function, as well as 
\begin{equation*}
	\mathcal{G}_{1,\infty} := \left\{ G\in \mathcal{G}_1\ :\ \lim_{\zeta\to\infty} \frac{\zeta G'(\zeta) - G(\zeta)}{\zeta} = \infty \right\}.
\end{equation*}
It readily follows from the definition of $\mathcal{G}_{1, \infty}$ and l'Hospital rule that any $G\in\mathcal{G}_{1, \infty}$ satisfies
\begin{equation}
	\lim_{\zeta \to \infty} G'(\zeta) = \lim_{\zeta\to \infty} \frac{G(\zeta)}{\zeta} = \infty. \label{G0:Cond0}
\end{equation} 

\begin{remark} \label{RMK3}
	One can check that the functions $z\mapsto z^m$ with $m\in (1,2]$ and $z\mapsto z \big[\ln\big(e^{m-1}+z\big)\big]^m$ with $m>1$ belong to $\mathcal{G}_{1,\infty}$, but not $z\mapsto \ln(1+z)$ (though it belongs to $\mathcal{G}_1$ and satisfies~\eqref{G0:Cond0}).
\end{remark}
 
With this notation, we may summarize the main contribution of this work as follows: given a collision kernel satisfying the symmetry and non-negativity condition~\eqref{ACond} and a fragment distribution function satisfying~\eqref{LMC1} and~\eqref{bCond}, we construct a global mild solution to~\eqref{FSNLBE} for any initial condition $\psi^{\rm{in}} \in Y_{1,+}$ such that
\begin{equation*}
	\sum_{i=1}^\infty G_0(i) \psi_i^{\rm{in}}< \infty
\end{equation*}
for some function $G_0\in \mathcal{G}_{1,\infty}$. In particular, the price to pay for having no growth condition on the collision kernel is that we cannot handle arbitrary initial data in $Y_{1,+}$. Still, the existence result applies to any initial condition $\psi^{\rm{in}}$ which belongs to $Y_{\sigma,+}$ for some $\sigma>1$.

Before presenting the existence result, we first recall the definition of a global mild solution to~\eqref{FSNLBE} in $Y_{1,+}$.

\begin{definition} \label{DEF}
Consider $\psi^{\rm{in}} \in  Y_{1,+}$. A global mild solution $\psi=(\psi_i)_{i\ge 1}$ to~\eqref{FSNLBE} is a sequence of non-negative functions in $L^{\infty}((0,\infty),Y_{1,+})$ satisfying, for each $i\ge 1$ and $t>0$,
\begin{itemize}
\item[(i)] $\psi_i\in C([0,\infty))$;
\item[(ii)] 
\begin{equation*}
	 \sum_{j=i+1}^{\infty} \sum_{k=1}^{\infty} \varphi_{i,j;k} \Gamma_{j,k} \psi_j \psi_k \in L^1((0,t)), \qquad  \sum_{j=1}^{\infty} \Gamma_{i,j} \psi_i \psi_j \in L^1((0,t));
\end{equation*}
\item[(iii)] 
\begin{equation}
\begin{split}
	\psi_i(t) & = \psi_i^{\rm{in}} + \int_0^t \sum_{j=i+1}^{\infty} \sum_{k=1}^{\infty} \varphi_{i,j;k} \Gamma_{j,k} \psi_j(s) \psi_k(s) ds \\
	& \qquad - (1-\delta_{i,1}) \int_0^t \sum_{j=1}^{\infty} \Gamma_{i,j} \psi_i(s) \psi_j(s) ds. 
\end{split} \label{IF}
\end{equation}
\end{itemize}
\end{definition}

 \begin{theorem} \label{TH1}
Assume that the kinetic coefficients $(\Gamma_{i,j})$  and $(\varphi_{i,j;k})$ satisfy the assumptions~\eqref{ACond}, \eqref{LMC1}, and~\eqref{bCond}. Consider $\psi^{\rm{in}}\in Y_{1,+}$ such that 
\begin{equation}
		\mathcal{J}_0 := \sum_{i=1}^{\infty} G_0(i) \psi_i^{\rm{in}} < \infty \label{G0:Cond1}
\end{equation}
for some $G_0\in \mathcal{G}_{1,\infty}$. Then there exists at least one global mild solution $\psi$ to~\eqref{FSNLBE} in the sense of Definition~\ref{DEF} which additionally satisfies the mass conservation~\eqref{MCC}
\begin{equation*}
	\|\psi(t)\|_1 = \|\psi^{\rm{in}}\|_1, \qquad t\ge 0.
\end{equation*}.

Moreover, for any non-negative sequence $(\Lambda_i)_{i\ge 1}$ such that the sequence $(\Lambda_i/i)_{i\ge 1}$ is non-decreasing and 
\begin{equation}
	M_{\Lambda}(\psi^{\mathrm{in}}) := \sum_{i = 1}^\infty \Lambda_i \psi_i^{\mathrm{in}} < \infty, \label{LambdaMom}
\end{equation}
then the mild solution $\psi = (\psi_i)_{i \ge 1}$ constructed above satisfies the uniform-in-time tail estimate
\begin{equation}
	\sum_{i = r}^\infty \Lambda_i \psi_i(t) \le \sum_{i = r}^\infty \Lambda_i \psi_i^{\mathrm{in}} \quad \text{for all } t > 0 \;\text{ and }\; r \ge 1. \label{TailIneq}
\end{equation}
\end{theorem}

We next identify an additional simple structure condition on the kinetic coefficients $(\Gamma_{i,j})_{i,j \ge 1}$ which, along with the boundedness of the fragment distribution function, allows us to show that the mild solution to~\eqref{SNLBE} constructed in \Cref{TH1} is actually a classical solution.

 \begin{theorem} \label{TH2}
 	Assume that there is a non-negative sequence $(\Lambda_i)_{i\ge 1}$ with $\Lambda_1\ge 1$ such that the sequence $(\Lambda_i/i)_{i\ge 1}$ is non-decreasing and the kinetic coefficients $(\Gamma_{i,j})$ satisfies
 	\begin{equation}
 		0 \le \Gamma_{i,j} = \Gamma_{j,i} \le \Lambda_i \Lambda_j, \qquad i,j\ge 1. \label{ACond1}
 	\end{equation}
 	Assume also that $(\varphi_{i,j;k})$ satisfy the assumptions~\eqref{LMC1} and~\eqref{bCond} with $\alpha_1=0$; that is,
 	\begin{equation}
 		\varphi_{i,j;k} \leq \alpha_0, \qquad 1 \leq i \leq j-1, \quad j \geq 2, \quad k \geq 1. \label{bCond3}
 	\end{equation}
 	Finally, assume that the initial condition $\psi^{\rm{in}}\in Y_{1,+}$ satisfies~\eqref{G0:Cond1} for some function $G_0\in\mathcal{G}_{1,\infty}$, as well as~\eqref{LambdaMom}. Then there is at least one mass conserving classical solution $\psi = (\psi_i)_{i \ge 1}$ to~\eqref{FSNLBE}; that is, for each $i\ge 1$, $\psi_i \in C^1([0,\infty))$,
	\begin{equation}
		(1-\delta_{i,1}) \sum_{j=1}^{\infty} \Gamma_{i,j} \psi_{i} \psi_j \in C([0,\infty)), \qquad \sum_{j=i+1}^{\infty} \sum_{k=1}^{\infty} \varphi_{i,j;k}\Gamma_{j,k} \psi_j \psi_k \in C([0,\infty)),\label{CNT}
\end{equation}
and~\eqref{SNLBE} is satisfied pointwise for all $i \ge 1$. Furthermore, $\psi$ satisfies 
\begin{equation}
	M_{\Lambda}(\psi(t)) := \sum_{i=1}^{\infty}\Lambda_i\psi_i(t) \le M_{\Lambda}(\psi^{\rm{in}}), \qquad t\ge 0. \label{HMB}
\end{equation}
\end{theorem}

The significant role of the structural assumption~\eqref{ACond1} in the existence theory for coagulation-fragmentation equations is already observed in \cite{PhL 1999, PhL 25}. Taking $\Lambda_i=\sqrt{A} i$, $i\ge 1$, the collision rate satisfies~\eqref{QG} and we recover the existence result established in \cite[Theorem~2.1]{AGP 24I} by a different approach, but for a smaller set of initial conditions.

We supplement \Cref{TH2} with a uniqueness result under the assumption of the finiteness of a higher moment of the initial condition. The proof relies on an estimate on the difference of two solutions in a suitably chosen weighted space, a classical approach to uniqueness for coagulation-fragmentation equations, see \cite[Section~8.2.5]{BLL 2019} and the references therein.
 
\begin{theorem} \label{TH3}
Assume that the hypotheses of Theorem \ref{TH2} hold, and let $\psi^{\mathrm{in}} \in Y_{1,+}$ satisfy
\begin{equation}
M_{\Lambda^2}(\psi^{\mathrm{in}}) := \sum_{i=1}^{\infty} \Lambda_i^2 \psi_i^{\mathrm{in}} < \infty. \label{UNIC}
\end{equation}
Then there exists a unique classical solution $\psi = (\psi_i)_{i \ge 1}$ to~\eqref{FSNLBE} such that
\begin{equation}
M_{\Lambda^2}(\psi(t)) := \sum_{i=1}^\infty \Lambda_i^2 \psi_i(t) \le M_{\Lambda^2}(\psi^{\mathrm{in}}), \quad \text{for all } t \ge 0. \label{UNME}
\end{equation}
\end{theorem}

\medskip

The paper is organized as follows. \Cref{SEC2} is dedicated to establishing the proof of \Cref{TH1}, employing a compactness method and the approximation of the system~\eqref{FSNLBE} by finite systems of ordinary differential equations. \Cref{SEC3} is devoted to establishing the existence of classical solutions, which is derived from \Cref{TH1} and appropriate moment estimates. The uniqueness of classical solutions is then dealt with in \Cref{SEC4}. In \Cref{SEC5}, we investigate the asymptotic behavior of solutions. Finally, \Cref{SEC6} presents the results of numerical simulations.

\section{\textbf{Existence of Mild Solutions}} \label{SEC2}
 
We fix $\psi^{\rm{in}}\in Y_{1,+}$ satisfying~\eqref{G0:Cond1} for some $G_0\in\mathcal{G}_{1,\infty}$. We then define 
\begin{align*}
	G_1(\zeta) =\frac{G_0(\zeta)}{\zeta}, \quad \zeta>0, \qquad G_1(0) =0,
\end{align*}
and recall that the properties of $G_0$ implies that $G_1$ is non-negative, increasing, and concave. In addition,
\begin{equation}
	\lim_{\zeta\to \infty} \zeta G_1'(\zeta) = \lim_{\zeta\to \infty} \frac{\zeta G_0'(\zeta) - G_0(\zeta)}{\zeta} = \infty,  \label{G0:Cond2}
\end{equation}
since $G_0\in\mathcal{G}_{1,\infty}$.

 For $p\ge 3$, let us next introduce the truncated version of~\eqref{FSNLBE} as
 \begin{subequations} \label{FTSNLBE}
 \begin{equation} 
 \frac{d\psi_i^p}{dt} = \sum_{j=i+1}^{p} \sum_{k=1}^{p}  \Gamma_{j,k} \varphi_{i,j;k} \psi_j^p \psi_k^p - (1-\delta_{i,1}) \sum_{j=1}^{p} \Gamma_{i,j} \psi_i^p \psi_j^p, \label{TSNLBE}
 \end{equation}
 \begin{equation}
 \psi_i^p(0) = \psi_i^{\rm{in}},\label{TSNLBEIC}
 \end{equation}
 \end{subequations}
where  $ i\in \{1,2,\cdots,p\}$. We first report the well-posedness of~\eqref{FTSNLBE} which is a classical consequence of the Cauchy-Lipschitz (or Picard-Lindel\"of) theorem, together with the mass conservation~\eqref{TMC}.

\begin{lemma} \label{LEM1}
Let $p \geq 3$. There is a unique solution 
$\psi^p = (\psi^p_i)_{1 \leq i \leq p} \in C^1\big([0,\infty), [0,\infty)^p\big)$ to~\eqref{FTSNLBE}. For any sequence of non-negative real numbers $(\upsilon_i)_{i \geq 1}$, there holds
\begin{equation}
	\frac{d}{dt} \sum_{i=1}^{p} \upsilon_i \psi^p_i + \sum_{j=2}^{p} \sum_{k=1}^{p} \Big(\upsilon_j - \sum_{i=1}^{j-1} \upsilon_i \varphi_{i,j;k}\Big) \Gamma_{j,k} \psi^p_j \psi^p_k = 0. \label{TR:WF}
\end{equation}
In particular,
\begin{equation}
	\sum_{i=1}^p i \psi_i^p(t) = \sum_{i=1}^p i \psi_i^{\rm{in}} \le \|\psi^{\rm{in}}\|_1, \qquad t\in [0,\infty). \label{TMC}
\end{equation}
\end{lemma}

\begin{proof}
	As already mentioned, the local well-posedness of~\eqref{FTSNLBE} and the componentwise non-negativity of the solution are straightforward consequences of the Cauchy-Lipschitz theorem since the right-hand side of~\eqref{FTSNLBE} is locally Lipschitz continuous and quasi-positive in $\mathbb{R}^p$. Next, a simple computation leads to the identity~\eqref{TR:WF} which, in turn, gives the mass conservation~\eqref{TMC} (with the choice $\upsilon_i=i$, $i\ge 1$) and thereby excludes finite time blowup of the solution.
\end{proof} 

Exploiting~\eqref{TR:WF} for a particular class of sequences $(\upsilon_i)_{i\ge 1}$ provides addition information on $\psi^p$.

\begin{lemma}\label{LEM1B}
Let $p\ge 3$ and consider a sequence of non-negative real numbers $(\upsilon_i)_{i \geq 1}$ such that $(\upsilon_i/i)_{i \geq 1}$ is non-decreasing. Then, for $t\ge 0$, 
\begin{subequations}
\begin{equation}
    0 \leq \sum_{i=1}^{p} \upsilon_i \psi^p_i (t) \leq \sum_{i=1}^{p} \upsilon_i \psi^{\rm{in}}_i\label{LEM1:EST1},
\end{equation}
\begin{equation}
    0 \leq \int_0^t \sum_{j=2}^{p} \sum_{k=1}^{p} \Bigg(\upsilon_j - \sum_{i=1}^{j-1}\upsilon_i \varphi_{i,j;k}\Bigg) \Gamma_{j,k} \psi^p_j (s) \psi^p_k (s) \, ds \leq \sum_{i=1}^{p} \upsilon_i \psi^{\rm{in}}_i. \label{LEM1:EST2}
\end{equation}
\end{subequations}
 \end{lemma}

\begin{proof}
Let $j\ge 2$ and $k\ge 1$. Since $(\upsilon_i)_{i\ge 1}$ is a non-negative sequence such that $(\upsilon_i/i)_{i\ge 1}$ is non-decreasing, we infer from~\eqref{LMC1} that
\begin{align*}
\upsilon_j-\sum_{i=1}^{j-1} \upsilon_i \varphi_{i,j;k}&= \upsilon_j- \sum_{i=1}^{j-1} \frac{\upsilon_i}{i} i \varphi_{i,j;k}\\
& \ge \upsilon_j-\frac{\upsilon_j}{j} \sum_{i=1}^{j-1} i \varphi_{i,j;k} = \upsilon_j - \upsilon_j = 0.
\end{align*}
Each term on the right hand side of~\eqref{TR:WF} is therefore non-negative and we obtain~\eqref{LEM1:EST1} and~\eqref{LEM1:EST2} after integration with respect to time.
\end{proof}

Several estimates on the solutions to~\eqref{FTSNLBE} can then be deduced from \Cref{LEM1}. Let us begin with a uniform control on the tail of the first moment.

 \begin{lemma} \label{LEM2}
 For $p\ge 3$, $2 \le r\le p$, and $t\ge 0$, 
 \begin{align*}
 \sum_{i=r}^p i \psi_i^p(t) \le \sum_{i=r}^p i \psi_i^{\rm{in}} \le \sum_{i=r}^\infty i \psi_i^{\rm{in}}.
 \end{align*}
\end{lemma}

\begin{proof}
We note that the sequence defined by 
\begin{align*}
\upsilon_i =0, \quad 1\le i \le r-1, \qquad \upsilon_i = i, \quad i\ge r,
\end{align*}
is non-negative and $(\upsilon_i/i)_{i \geq 1}$ is non-decreasing. We then apply \Cref{LEM1} with this sequence and deduce \Cref{LEM2} from~\eqref{LEM1:EST1}.
\end{proof}

We next turn to estimates on the reaction terms.

\begin{lemma}\label{LEM3}
For $p> i\ge 1$ , there is $C_i := \mathcal{J}_0/[i G_1'(i+1)]>0$ such that
\begin{subequations}
\begin{equation}
	(1-\delta_{i,1}) \int_0^t \sum_{j=1}^p \Gamma_{i,j} \psi_i^p(s) \psi_j^p(s) ds \le C_i, \quad t\ge 0, \label{LEM3:EST1}
\end{equation}
\begin{equation}
\int_0^t \sum_{j=i+1}^p \sum_{k=1}^p \varphi_{i,j;k} \Gamma_{j,k} \psi_j^p(s) \psi_k^p(s) ds \le C_i, \quad t\ge 0. \label{LEM3:EST2}
\end{equation}
\end{subequations}
\end{lemma}

\begin{proof}
Since $(G_1(i))_{i\ge 1}= \Big(G_0(i)/i\Big)_{i\ge 1}$ is a non-decreasing sequence, we infer from~\eqref{G0:Cond1} and \Cref{LEM1} with $\upsilon_i = G_0(i)$, $i\ge 1$, that, for $t\ge 0$,
\begin{equation}
	\sum_{i=1}^p G_0(i) \psi_i^p(t) \le \mathcal{J}_0   \label{G:EST0}
\end{equation}
and
\begin{equation*}
	0\le \int_0^t \sum_{j=2}^p \sum_{k=1}^p \Bigg(G_0(j) - \sum_{i=1}^{j-1} G_0(i) \varphi_{i,j;k}\Bigg) \Gamma_{j,k} \psi_j^p(s) \psi_k^p(s) ds \le \mathcal{J}_0,
\end{equation*}
which also reads, according to~\eqref{LMC1} and the definition of $G_1$,
\begin{equation}
	0\le \int_0^t \sum_{j=2}^p \sum_{k=1}^p \sum_{i=1}^{j-1} (G_1(j) -  G_1(i))i \varphi_{i,j;k} \Gamma_{j,k} \psi_j^p(s) \psi_k^p(s) ds \le \mathcal{J}_0. \label{G:EST1}
\end{equation}
Owing to the concavity of $G_1$, 
\begin{equation}
	G_1(j) - G_1(l) \ge (j-l) G_1'(j), \quad 1\le l\le j-1, \label{G:ESTC}
\end{equation}
so that, using once more~\eqref{LMC1}, we have 
\begin{align*}
	\sum_{i=1}^{j-1} \big(G_1(j) - G_1(i)\big) i \varphi_{i,j;k} &\ge \sum_{i=1}^{j-1} i(j - i) G_1'(j) \varphi_{i,j;k} \\
	& = G_1'(j) \left( j^2 - \sum_{i=1}^{j-1} i^2 \varphi_{i,j;k} \right) \\
	& \ge G_1'(j) \left( j^2 - (j-1) \sum_{i=1}^{j-1} i \varphi_{i,j;k} \right) \\
	& = j G_1'(j).
\end{align*}
Combining the above lower bound with~\eqref{G:EST1}, we find
\begin{equation}
	\int_0^t \sum_{j=2}^p \sum_{k=1}^p j G_1'(j) \Gamma_{j,k} \psi_j^p(s) \psi_k^p(s) \, ds \le \mathcal{J}_0. \label{G:EST2}
\end{equation}
In particular, for $p\ge i\ge 2$ and $t>0$, we deduce from~\eqref{G:EST2} that
\begin{equation*}
	\mathcal{J}_0 \ge \int_0^t \sum_{k=1}^p i G_1'(i) \Gamma_{i,k} \psi_i^p(s) \psi_k^p(s) \, ds,
\end{equation*}
whence~\eqref{LEM3:EST1} since $G_1'(i) \ge G_1'(i+1)$.

We next infer from~\eqref{G:EST1}, \eqref{G:ESTC}, and the monotonicity of $G_1$ that, for $p\ge i\ge 1$ and $t\ge 0$, 
\begin{align*}
	\mathcal{J}_0 \ge & \int_0^t \sum_{j=i+1}^p \sum_{k=1}^p \sum_{l=1}^{j-1} (G_1(j)-G_1(l)) l \varphi_{l,j;k} \Gamma_{j,k} \psi_j^p(s) \psi_k^p(s) ds\\
	&\ge  \int_0^t \sum_{j=i+1}^p \sum_{k=1}^p (G_1(j)-G_1(i)) i \varphi_{i,j;k} \Gamma_{j,k} \psi_j^p(s) \psi_k^p(s) ds\\
	&\ge i [G_1(i+1)- G_1(i)] \int_0^t \sum_{j=i+1}^p \sum_{k=1}^p \varphi_{i,j;k} \Gamma_{j,k} \psi_j^p(s) \psi_k^p(s) ds\\
	&\ge i G_1'(i+1) \int_0^t \sum_{j=i+1}^p \sum_{k=1}^p \varphi_{i,j;k} \Gamma_{j,k} \psi_j^p(s) \psi_k^p(s) ds,
\end{align*}
which completes the proof.
\end{proof}

An immediate consequence of \Cref{LEM3} is the following bound on the time derivative of $\psi_i^p$.

\begin{corollary}\label{Cor1}
For  $p\ge i \ge 1$ and $t\ge 0$, we have
\begin{equation}
	\int_0^t \Bigg| \frac{d\psi_i^p}{dt}(s)\Bigg| ds \leq 2C_i. \label{DERVBOUND}
\end{equation}
\end{corollary}

We now turn to the derivation of additional estimates on the reaction terms. Following the strategy outlined in \cite[Proof of Theorem~2.1]{AGP 24I}, we arrive at the main estimate of this section, which provides control over the contributions arising from the tails of the two infinite series on the right-hand side of~\eqref{IF}.

\begin{proposition}\label{Prop1}
For $i\ge 1$, $t>0$ and $i\le m<p$,
\begin{equation}
	(1-\delta_{i,1}) \int_0^t \sum_{j=m+1}^p \Gamma_{i,j}\psi_i^p(s) \psi_j^p(s) ds \le  \varepsilon_{m}, \label{Prop1:EST1}
\end{equation}
\begin{align}
	& \int_0^t \sum_{j=m+1}^p \sum_{k=1}^p \varphi_{i,j;k} \Gamma_{j,k} \psi_j^p(s) \psi_k^p(s) ds \le \omega_m(i), \label{G:EST4} \\
	& \int_0^t \sum_{j=i+1}^m \sum_{k=m+1}^p \varphi_{i,j;k} \Gamma_{j,k} \psi_j^p(s) \psi_k^p(s) ds \le \alpha_0 \varepsilon_m + \alpha_1 \omega_m(i), \label{Prop1:EST2}
\end{align}
with 
\begin{equation}
	\varepsilon_{m} := \frac{\mathcal{J}_0}{\displaystyle \inf\limits_{z \geq m} \{zG_1'(z)\}} \;\text{ and }\; \omega_m(i) := \frac{\alpha_1 \mathcal{J}_0}{G_1(m+1)-G_1(i)}. \label{VEM}
\end{equation}
\end{proposition}

\begin{proof}
 Consider first $i\ge 2$ and $i\le m<p$. Then, by~\eqref{G:EST2},
\begin{align*}
	\mathcal{J}_0 \ge& \int_0^t \sum_{j=2}^p j G_1'(j) \Gamma_{j,i} \psi_j^p(s) \psi_i^p(s) ds\\
	&\ge \int_0^t \sum_{j=m+1}^p jG_1'(j) \Gamma_{i,j} \psi_i^p(s) \psi_j^p(s) ds\\
	&\ge \inf_{z\ge m} \{z G_1'(z)\} \int_0^t \sum_{j=m+1}^p \Gamma_{i,j} \psi_i^p(s) \psi_j^p(s) ds,
\end{align*}
and we have proved~\eqref{Prop1:EST1}.

Next, for $1\le i \le m <p$, we infer from~\eqref{G:EST1} and the monotonicity of $G_1$ that
\begin{align*}
	\mathcal{J}_0 \ge& \int_0^t \sum_{j=m+1}^p \sum_{k=1}^p \sum_{l=1}^{j-1} (G_1(j)-G_1(l)) l \varphi_{l,j;k} \Gamma_{j,k} \psi_j^p(s) \psi_k^p(s) ds \\
	&\ge \int_0^t \sum_{j=m+1}^p \sum_{k=1}^p (G_1(j)-G_1(i)) i \varphi_{i,j;k} \Gamma_{j,k} \psi_j^p(s) \psi_k^p(s) ds \\
	&\ge i[G_1(m+1)-G_1(i)] \int_0^t \sum_{j=m+1}^p \sum_{k=1}^p \varphi_{i,j;k} \Gamma_{j,k} \psi_j^p(s) \psi_k^p(s) ds,
\end{align*}
which gives~\eqref{G:EST4}. Now, by~\eqref{bCond},
\begin{align*}
	\int_0^t \sum_{j=i+1}^m \sum_{k=m+1}^p \varphi_{i,j;k} \Gamma_{j,k} \psi_j^p(s) \psi_k^p(s) ds \le \alpha_0& \int_0^t \sum_{j=i+1}^m \sum_{k=m+1}^p \Gamma_{j,k} \psi_j^p(s) \psi_k^p(s) ds \\
	& + \alpha_1 \int_0^t\sum_{j=i+1}^m \sum_{k=m+1}^p \varphi_{i,k;j} \Gamma_{j,k} \psi_j^p(s) \psi_k^p(s) ds. 
\end{align*}
We estimate each of the terms on the right-hand side separately and begin with the first one. Owing to the monotonicity of $G_1$,
\begin{align*}
	\alpha_0 \int_0^t \sum_{j=i+1}^m \sum_{k=m+1}^p \Gamma_{j,k}& \psi_j^p(s) \psi_k^p(s) ds = \alpha_0 \int_0^t \sum_{j=i+1}^m \sum_{k=m+1}^p \frac{kG_1'(k)}{kG_1'(k)} \Gamma_{j,k} \psi_j^p(s) \psi_k^p(s) ds \\
	&\le \frac{\alpha_0}{\displaystyle \inf\limits_{z \geq m} \{zG_1'(z)\}} \int_0^t \sum_{j=i+1}^m \sum_{k=m+1}^p kG_1'(k) \Gamma_{j,k} \psi_j^p(s) \psi_k^p(s) ds \\
	&\le \frac{\alpha_0}{\displaystyle \inf\limits_{z \geq m} \{zG_1'(z)\}} \int_0^t \sum_{j=1}^p \sum_{k=2}^p kG_1'(k) \Gamma_{j,k} \psi_j^p(s) \psi_k^p(s) ds.
\end{align*}
Applying the substitution $j\leftrightarrow k$ and invoking~\eqref{G:EST2}, we obtain
\begin{equation*}
	\alpha_0 \int_0^t \sum_{j=i+1}^m \sum_{k=m+1}^p \Gamma_{j,k} \psi_j^p(s) \psi_k^p(s) ds \le \frac{\alpha_0\mathcal{J}_0}{\displaystyle \inf\limits_{z \geq m} \{zG_1'(z)\}}. 
\end{equation*}

We next deal with the second term and use once more the monotonicity of $G_1$ to obtain
\begin{align*}
	\alpha_1 & \int_0^t \sum_{j=i+1}^m \sum_{k=m+1}^p \varphi_{i,k;j}  \Gamma_{j,k} \psi_j^p(s) \psi_k^p(s) ds \\
	&= \alpha_1 \int_0^t \sum_{j=i+1}^m \sum_{k=m+1}^p \frac{G_1(k) - G_1(i)}{G_1(k) - G_1(i)} \varphi_{i,k;j} \Gamma_{j,k} \psi_j^p(s) \psi_k^p(s) ds \\
	&\le \frac{\alpha_1}{[G_1(m+1) - G_1(i)]} \int_0^t \sum_{j=i+1}^m \sum_{k=m+1}^p \sum_{l=1}^{k-1} [G_1(k) - G_1(l)] \varphi_{l,k;j} \Gamma_{j,k} \psi_j^p(s) \psi_k^p(s) ds \\
	& \le \frac{\alpha_1}{[G_1(m+1) - G_1(i)]} \int_0^t \sum_{j=1}^p \sum_{k=2}^p \sum_{l=1}^{k-1} [G_1(k) - G_1(l)] l \varphi_{l,k;j} \Gamma_{j,k} \psi_j^p(s) \psi_k^p(s) ds.
\end{align*}
Applying estimate~\eqref{G:EST1}, we end up with
\begin{equation*}
	\alpha_1 \int_0^t \sum_{j=i+1}^m \sum_{k=m+1}^p \varphi_{i,k;j} \Gamma_{j,k} \psi_j^p(s) \psi_k^p(s) ds \le \frac{\alpha_1 \mathcal{J}_0}{[G_1(m+1) - G_1(i)]}, 
\end{equation*}
which completes the proof.
\end{proof}

\begin{proof}[Proof of \Cref{TH1}] Owing to~\eqref{TMC} and~\eqref{DERVBOUND}, we are in a position to apply Helly’s selection principle \cite[Theorem~2.35]{Leoni 2009}, combined with a diagonal process, and conclude that there exist a subsequence of $(\psi^p)_{p \geq 3}$, still denoted by $(\psi^p)_{p \geq 3}$, and a sequence of functions $\psi = (\psi_i)_{i \geq 1}$ such that
\begin{equation} \label{Psi:Lim}
    \lim_{p \to \infty} \psi^{p}_i (t) = \psi_i (t) \quad \text{for all } t \geq 0 \text{ and } i \geq 1.
\end{equation}
Clearly, $\psi_i(t) \geq 0$ for $i \geq 1$ and $t \geq 0$, and it follows from~\eqref{TMC}, \Cref{LEM2}, and~\eqref{Psi:Lim} that, for $r\ge 1$, $p\ge q \ge r+1$ and $t \geq 0$,
\begin{equation*}
	\sum_{i=r}^q i \psi_i(t) = \lim_{p \to \infty} \sum_{i=r}^q i \psi_i^p(t) \le \sum_{i=r}^\infty i \psi_i^{\rm{in}}.
\end{equation*}
By letting $q \to \infty$, we obtain
\begin{equation}
	\sum_{i=r}^\infty i \psi_i(t) \le \sum_{i=r}^\infty i \psi_i^{\rm{in}}, \quad r\ge 1, \ t \geq 0. \label{Psi:L1Lim}
\end{equation}
A similar argument allows us to deduce that 
\begin{equation}
	\sum_{i=1}^{\infty} G_0(i) \psi_i(t) \le \mathcal{J}_0,  \quad t\ge 0, \label{THM1:EST5}
\end{equation}
from~\eqref{G:EST0} and~\eqref{Psi:Lim}. Next, let $t\ge 0$, $i\ge 1$, and $r > i$. It follows from \Cref{LEM3} and~\eqref{Psi:Lim} that 
\begin{equation*}
	(1-\delta_{i,1}) \int_0^t \sum_{j=1}^r \Gamma_{i,j} \psi_i(s) \psi_j(s) ds =(1-\delta_{i,1}) \lim_{p\to \infty}  \int_0^t \sum_{j=1}^r \Gamma_{i,j} \psi_i^p(s) \psi_j^p(s) ds \le C_i
\end{equation*}
and
\begin{equation*}
	\int_0^t \sum_{j=i+1}^r \sum_{k=1}^r \varphi_{i,j;k} \Gamma_{j,k} \psi_j(s) \psi_k(s) ds = \lim_{p\to \infty} \int_0^t \sum_{j=i+1}^r \sum_{k=1}^r \varphi_{i,j;k} \Gamma_{j,k} \psi_j^p(s) \psi_k^p(s) ds \le C_i.
\end{equation*}
Letting $r \to \infty$ and using Fatou's lemma, we obtain
\begin{equation}
	(1-\delta_{i,1}) \int_0^t \sum_{j=1}^{\infty} \Gamma_{i,j} \psi_i(s) \psi_j(s) ds \le C_i, \quad t\ge 0, \label{THM1:EST6}
\end{equation}
and
\begin{equation}
	\int_0^t \sum_{j=i+1}^{\infty} \sum_{k=1}^{\infty} \varphi_{i,j;k} \Gamma_{j,k} \psi_j(s) \psi_k(s) ds \le C_i, \quad t\ge 0. \label{THM1:EST7}
\end{equation}
In the same way, we infer from \Cref{Prop1} and~\eqref{Psi:Lim} that, for $m\ge i \ge 1$,
\begin{equation}
	(1-\delta_{i,1}) \int_0^t \sum_{j=m+1}^\infty \Gamma_{i,j} \psi_i(s) \psi_j(s) ds \le \varepsilon_{m}  \label{THM1:EST1}
\end{equation}
and
\begin{align}
	& \int_0^t \sum_{j=m+1}^\infty \sum_{k=1}^\infty \varphi_{i,j;k} \Gamma_{j,k} \psi_j(s) \psi_k(s) ds \le \omega_m(i), \label{THM1:EST3} \\
	& \int_0^t \sum_{j=i+1}^m \sum_{k=m+1}^\infty \varphi_{i,j;k} \Gamma_{j,k} \psi_j(s) \psi_k(s) ds \le \alpha_0\varepsilon_{m} + \alpha_1 \omega_m(i). \label{THM1:EST2}
\end{align}

We are now ready to complete the proof of \Cref{TH1} and first observe that it follows from the estimates~\eqref{THM1:EST6} and~\eqref{THM1:EST7} that \Cref{DEF}(ii) is satisfied. Besides, the convergence~\eqref{Psi:Lim} and the tail control~\eqref{Psi:L1Lim} imply that, for any $t\ge 0$ and $p\ge r\ge 2$,
\begin{align*}
	\| \psi^p(t) - \psi(t) \|_1 \le & \sum_{j=1}^{r-1} i |\psi_j^p(t) - \psi_i(t)| + \sum_{i=r}^\infty i \psi_i(t) + \sum_{i=r}^p i \psi_i^p(t) \\
	& \le \sum_{j=1}^{r-1} i |\psi_j^p(t) - \psi_i(t)| + 2\sum_{i=r}^\infty i \psi_i^{\rm{in}}.
\end{align*}
Consequently, 
\begin{equation*}
	\limsup_{p\to\infty} \| \psi^p(t) - \psi(t) \|_1 \le 2\sum_{i=r}^\infty i \psi_i^{\rm{in}},
\end{equation*}
and we take the limit $r\to\infty$ in the above inequality and use the summability properties of $\psi^{\rm{in}}$ to conclude that
\begin{equation}
	\lim_{p \to \infty} \| \psi^p(t) - \psi(t) \|_1 = 0. \label{THM1:EST10}
\end{equation} 
In particular, we infer from~\eqref{TMC} and~\eqref{THM1:EST10} that, for all $t \geq 0$, we have  
\begin{equation*}
	\| \psi(t) \|_1 = \lim_{p \to \infty} \| \psi^p(t) \|_1 = \lim_{p \to \infty} \| \psi^p(0) \|_1 = \| \psi^{\rm{in}} \|_1,
\end{equation*} 
which shows that $\psi$ satisfies the mass conservation property~\eqref{TMC}.

\noindent We next fix $i \ge 1$ and $t\ge 0$. On the one hand, it follows from~\eqref{Prop1:EST1}, \eqref{Psi:Lim}, and~\eqref{THM1:EST1} that, for $i+1\le m< p$, we have
\begin{align*}
	(1-\delta_{i,1}) \int_0^t & \Bigg| \sum_{j=1}^{p} \Gamma_{i,j} \psi_i^p(s) \psi_j^p(s) - \sum_{j=1}^{\infty} \Gamma_{i,j} \psi_i(s) \psi_j(s)\Bigg| ds\\
	&\le \int_0^t \sum_{j=1}^{m} \Gamma_{i,j} |\psi_i^p(s) \psi_j^p(s) - \psi_i(s) \psi_j(s)| ds + \int_0^t \sum_{j=m+1}^{p} \Gamma_{i,j} \psi_i^p(s) \psi_j^p(s) ds\\
	&\qquad + \int_0^t \sum_{j=m+1}^{\infty} \Gamma_{i,j} \psi_i(s) \psi_j(s) ds \\
	& \le \int_0^t \sum_{j=1}^{m} \Gamma_{i,j} |\psi_i^p(s) \psi_j^p(s) - \psi_i(s) \psi_j(s)| ds + 2 \varepsilon_m,
\end{align*}
recalling that $\varepsilon_{m}$ is defined in~\eqref{VEM}. By virtue of~\eqref{TMC}, \eqref{Psi:Lim}, and~\eqref{Psi:L1Lim}, together with the Lebesgue dominated convergence theorem, we may pass to the limit as $p \to \infty$ in the above inequality to obtain
\begin{equation*}
	\limsup_{p\to \infty}\ (1-\delta_{i,1}) \int_0^t \Bigg| \sum_{j=1}^{p} \Gamma_{i,j} \psi_i^p(s) \psi_j^p(s) - \sum_{j=1}^{\infty} \Gamma_{i,j} \psi_i(s) \psi_j(s)\Bigg| ds \le 2 \varepsilon_m.
\end{equation*}
Since $G_0\in\mathcal{G}_{1,\infty}$, the function $G_1$ satisfies~\eqref{G0:Cond2}, so that the right hand side of the above inequality converges to zero as $m\to \infty$. We thus let $m\to \infty$ in the above inequality to conclude that
\begin{equation}
	\lim_{p\to \infty} (1-\delta_{i,1}) \int_0^t\Bigg| \sum_{j=1}^{p} \Gamma_{i,j} \psi_i^p(s) \psi_j^p(s) - \sum_{j=1}^{\infty} \Gamma_{i,j} \psi_i(s) \psi_j(s) \Bigg| ds=0. \label{IF:ST}
\end{equation}
We next turn to the convergence of the first term on the right hand side of~\eqref{TSNLBE}. For $p>m\ge i+1$, we infer from~\eqref{Prop1:EST2}, \eqref{THM1:EST3}, and~\eqref{THM1:EST2} that
\begin{align}
	\int_0^t\Bigg|\sum_{j=i+1}^{p} \sum_{k=1}^{p} & \varphi_{i,j;k} \Gamma_{j,k} \psi_j^p(s) \psi_k^p(s) - \sum_{j=i+1}^{\infty} \sum_{k=1}^{\infty} \varphi_{i,j;k} \Gamma_{j,k} \psi_j(s) \psi_k(s) \Bigg| ds \nonumber \\
	&\le \int_0^t \Bigg| \sum_{j=i+1}^{m} \sum_{k=1}^m \varphi_{i,j;k} \Gamma_{j,k} \big(\psi_j^p(s) \psi_k^p(s) - \psi_j(s) \psi_k(s) \big)\Bigg| ds \nonumber \\
	& \quad + \int_0^t \sum_{j=i+1}^m \sum_{k=m+1}^{p} \varphi_{i,j;k} \Gamma_{j,k} \psi_j^p(s) \psi_k^p(s) ds \nonumber \\ 
	& \quad +\int_0^t \sum_{j=m+1}^p \sum_{k=1}^{p} \varphi_{i,j;k} \Gamma_{j,k} \psi_j^p(s) \psi_k^p(s) ds \nonumber \\
	& \quad + \int_0^t \sum_{j=i+1}^m \sum_{k=m+1}^{\infty}  \varphi_{i,j;k} \Gamma_{j,k} \psi_j(s) \psi_k(s) ds \nonumber \\
	& \quad +\int_0^t \sum_{j=m+1}^{\infty} \sum_{k=1}^{\infty} \varphi_{i,j;k} \Gamma_{j,k} \psi_j(s) \psi_k(s) ds \nonumber \\
	& \le \int_0^t \Bigg| \sum_{j=i+1}^{m} \sum_{k=1}^m \varphi_{i,j;k} \Gamma_{j,k} \big(\psi_j^p(s) \psi_k^p(s) - \psi_j(s) \psi_k(s) \big)\Bigg| ds \label{THM1:EST13} \\
	& \quad + 2 \big( \alpha_0\varepsilon_{m} + \alpha_1 \omega_m(i) \big) +  2 \omega_m(i). \nonumber
\end{align}
On the one hand, it follows from~\eqref{TMC}, \eqref{Psi:Lim}, \eqref{Psi:L1Lim} and the Lebesgue dominated convergence theorem that
\begin{equation}
	\lim_{p\to \infty} \int_0^t\Bigg| \sum_{j=i+1}^p \sum_{k=1}^p \varphi_{i,j,k} \Gamma_{j,k}(\psi_j^p \psi_k^p - \psi_j \psi_k) \Bigg |d\tau =0. \label{THM1:EST11}
\end{equation}
On the other hand, it follows from the properties~\eqref{G0:Cond0} and~\eqref{G0:Cond2} that
\begin{equation}
	 \lim_{m\to \infty} \varepsilon_{m} = \lim_{m\to\infty} \omega_m(i) = 0. \label{THM1:EST12}
\end{equation}
Consequently, using~\eqref{THM1:EST11}, we may take the limit $p\to \infty$ in~\eqref{THM1:EST13} and find
\begin{align*}
	\limsup_{p\to \infty}\int_0^t\Bigg|\sum_{j=i+1}^{p} \sum_{k=1}^{p} \varphi_{i,j;k} \Gamma_{j,k} \psi_j^p(s) \psi_k^p(s) & - \sum_{j=i+1}^{\infty} \sum_{k=1}^{\infty} \varphi_{i,j;k} \Gamma_{j,k} \psi_j(s) \psi_k(s) \Bigg| ds \\ & \le 2 \big( \alpha_0\varepsilon_{m} + \alpha_1 \omega_m(i) \big) +  2 \omega_m(i).
\end{align*}
We next let $m\to \infty$ and deduce from~\eqref{THM1:EST12} that
\begin{equation}
	\lim_{p\to \infty}\int_0^t\Bigg|\sum_{j=i+1}^{p} \sum_{k=1}^{p} \varphi_{i,j;k} \Gamma_{j,k} \psi_j^p(s) \psi_k^p(s) - \sum_{j=i+1}^{\infty} \sum_{k=1}^{\infty} \varphi_{i,j;k} \Gamma_{j,k} \psi_j(s) \psi_k(s) \Bigg| ds = 0. \label{IF:FT}
\end{equation}
Owing to~ \eqref{Psi:Lim}, \eqref{IF:ST}, and~\eqref{IF:FT}, we may pass to the limit as $p \to \infty$ in the integral formulation of the equation satisfied by $\psi_i^p$, thereby concluding that $\psi_i$ satisfies~\eqref{IF}. Moreover, since the right-hand side of~\eqref{IF} lies in $L^1_{\mathrm{loc}}([0, \infty))$, it follows that $\psi_i \in C([0, \infty))$, and we have proved the existence part of \Cref{TH1}. 

\medskip

Finally, we assume that the initial condition $ \psi^{\rm{in}}$ satisfies~\eqref{LambdaMom}. Let $ r\ge 1$ and $t\ge 0$. Since the sequence defined by $\upsilon_i = 0$ for $1 \leq i \leq r-1$ and $\upsilon_i = \Lambda_i$ for $i \geq r$ satisfies the assumptions of \Cref{LEM1B}, we infer from~\eqref{LEM1:EST1} that, for $p > m > r$,  the solution $\psi^p$ to~\eqref{FTSNLBE} satisfies
\begin{equation*}
	\sum_{i=r}^m \Lambda_i \psi_i^{p}(t) \leq \sum_{i=r}^p \Lambda_i \psi_i^{p}(t) \leq \sum_{i=r}^p \Lambda_i \psi_i^{\mathrm{in}} \leq \sum_{i=r}^{\infty} \Lambda_i \psi_i^{\mathrm{in}}.
\end{equation*}
Thanks to~\eqref{Psi:Lim}, we first take the limit as $p \to \infty$ and subsequently let $m \to \infty$  in the above inequality to derive~\eqref{TailIneq} and complete the proof of \Cref{TH1}.
\end{proof}

\section{\textbf{Existence of Classical Solutions}}\label{SEC3}

\begin{proof}[Proof of \Cref{TH2}]
In this section, we assume that the kinetic coefficients $(\Gamma_{i,j})$, $(\varphi_{i,j,k})$, and the initial condition $\psi^{\rm{in}}$ satisfy~\eqref{ACond1}, \eqref{LMC1}, \eqref{bCond3}, \eqref{G0:Cond1}, and~\eqref{LambdaMom}, respectively. It follows from \Cref{TH1} that~\eqref{FSNLBE} admits a global mass-conserving mild solution satisfying~\eqref{TailIneq}. Our goal is to demonstrate that the additional assumptions~\eqref{ACond1} and~\eqref{bCond3} ensure the continuity properties~\eqref{CNT} as well as the $C^1$-regularity of $\psi_i$ for each $i\ge 1$. To this end, for $(s, t) \in [0, \infty)^2$ and $m \ge i \ge 2$, we deduce from~\eqref{MCC}, \eqref{TailIneq} and~\eqref{ACond1} that
\begin{align*}
	\Big| \sum_{j=1}^{\infty} \Gamma_{i,j} & \psi_i(t) \psi_j(t) - \sum_{j=1}^{\infty} \Gamma_{i,j} \psi_i(s) \psi_j(s)\Big| \le \sum_{j=1}^{\infty} \Gamma_{i,j} |\psi_i(t) \psi_j(t) - \psi_i(s) \psi_j(s)|\\
	&\le \sum_{j=1}^{m}\Gamma_{i,j} |(\psi_i \psi_j)(t) - (\psi_i \psi_j)(s)| + \Lambda_i \sum_{j=m+1}^{\infty} \Lambda_j [(\psi_i\psi_j)(t)+ (\psi_i\psi_j)(s)]\\
	& \le \sum_{j=1}^{m}\Gamma_{i,j}|(\psi_i\psi_j)(t) - (\psi_i\psi_j)(s)| + 2 \Lambda_i \|\psi^{\rm{in}}\|_1\sum_{j=m+1}^{\infty} \Lambda_j \psi_j^{\rm{in}}. 
\end{align*}
Due to the continuity of $\psi_j$ for all $j \ge 1$,
\begin{align*}
	\limsup_{s\to t} \Big| \sum_{j=1}^{\infty} \Gamma_{i,j} \psi_i(t) \psi_j(t) - \sum_{j=1}^{\infty} \Gamma_{i,j} \psi_i(s) \psi_j(s)\Big| \le 2 \Lambda_i \|\psi^{\rm{in}}\|_1\sum_{j=m+1}^{\infty} \Lambda_j \psi_j^{\rm{in}},
\end{align*}
and, applying~\eqref{LambdaMom}, we can let $m \to \infty$ in the above inequality to conclude that
\begin{equation}
	\lim_{s\to t} \sum_{j=1}^{\infty} \Gamma_{i,j} \psi_i(s) \psi_j(s) = \sum_{j=1}^{\infty} \Gamma_{i,j} \psi_i(t) \psi_j(t). \label{FTC}
\end{equation}

Similarly, consider $(s,t) \in [0, \infty)^2$ and $m > i \ge 1$. By making use of~\eqref{LMC1}, \eqref{LambdaMom}, \eqref{TailIneq}, and~\eqref{ACond1}, we obtain
\begin{align*}
	\Bigg|\sum_{j=i+1}^{\infty} \sum_{k=1}^{\infty} & \varphi_{i,j;k} \Gamma_{j,k} \psi_j(t) \psi_k(t) -\sum_{j=i+1}^{\infty} \sum_{k=1}^{\infty} \varphi_{i,j;k} \Gamma_{j,k} \psi_j(s) \psi_k(s) \Bigg|\\
	& \le \sum_{j=i+1}^{\infty} \sum_{k=1}^{\infty} \varphi_{i,j,k} \Gamma_{j,k} |(\psi_j\psi_k)(t) - (\psi_j\psi_k)(s)| \\
	& \le \sum_{j=i+1}^{m} \sum_{k=1}^{m} \varphi_{i,j;k} \Gamma_{j,k} |(\psi_j\psi_k)(t) - (\psi_j\psi_k)(s)| \\
	& \quad + \sum_{j=i+1}^{m} \sum_{k=m+1}^{\infty} \varphi_{i,j;k} \Gamma_{j,k} [(\psi_j \psi_k)(t)+ (\psi_j \psi_k)(s)] \\
	& \quad + \sum_{j=m+1}^{\infty} \sum_{k=1}^{\infty} \varphi_{i,j;k} \Gamma_{j,k} [(\psi_j \psi_k)(t)+ (\psi_j \psi_k)(s)|.
\end{align*}
Thanks to~\eqref{TailIneq}, \eqref{ACond1}, and~\eqref{bCond3}, 
\begin{align*}
	\sum_{j=i+1}^{m} \sum_{k=m+1}^{\infty} & \varphi_{i,j;k} \Gamma_{j,k} [(\psi_j \psi_k)(t)+ (\psi_j \psi_k)(s)]\\
	& \le \alpha_0 \sum_{j=i+1}^{m} \sum_{k=m+1}^{\infty} \Lambda_j \Lambda_k[(\psi_j \psi_k)(t)+ (\psi_j \psi_k)(s)]\\
	& \le \alpha_0 \Bigg(M_{\Lambda}(\psi(t)) \sum_{k=m+1}^{\infty} \Lambda_k \psi_k(t) + M_{\Lambda}(\psi(s)) \sum_{k=m+1}^{\infty} \Lambda_k \psi_k(s)\Bigg)\\
	& \le 2\alpha_0 M_{\Lambda}(\psi^{\rm{in}})\sum_{k=m+1}^{\infty} \Lambda_k \psi_k^{\rm{in}}.
\end{align*}
In a similar way we can evaluate 
\begin{equation*}
	\sum_{j=m+1}^{\infty} \sum_{k=1}^{\infty} \varphi_{i,j;k} \Gamma_{j,k} [(\psi_j \psi_k)(t)+ (\psi_j \psi_k)(s) \le 2\alpha_0 M_{\Lambda}(\psi^{\rm{in}}) \sum_{j=m+1}^{\infty} \Lambda_j \psi_j^{\rm{in}}.
\end{equation*}
Gathering the above estimates, we finally obtain
\begin{align*}
	\Bigg|\sum_{j=i+1}^{\infty} \sum_{k=1}^{\infty} & \varphi_{i,j,k} \Gamma_{j,k} \psi_j(t) \psi_k(t) -\sum_{j=i+1}^{\infty} \sum_{k=1}^{\infty} \varphi_{i,j,k} \Gamma_{j,k} \psi_j(s) \psi_k(s) \Bigg|\\
	& \le \sum_{j=i+1}^{m} \sum_{k=1}^{m} \varphi_{i,j,k} \Gamma_{j,k} |(\psi_j\psi_k)(t) - (\psi_j\psi_k)(s)| + 4 \alpha_0 M_{\Lambda}(\psi^{\rm{in}}) \sum_{j=m+1}^{\infty} \Lambda_j \psi_j^{\rm{in}}
\end{align*}
and proceed as in the derivation of~\eqref{FTC} to find
\begin{align}
\lim_{s\to t}\sum_{j=i+1}^{\infty} \sum_{k=1}^{\infty} \varphi_{i,j,k} \Gamma_{j,k} \psi_j(s) \psi_k(s)= 
\sum_{j=i+1}^{\infty} \sum_{k=1}^{\infty} \varphi_{i,j,k} \Gamma_{j,k} \psi_j(t) \psi_k(t). \label{STC}
\end{align}
The continuity properties~\eqref{CNT} follow directly from~\eqref{FTC} and~\eqref{STC}. Moreover, applying \Cref{DEF}(iii) in conjunction with~\eqref{CNT}, we deduce that $\psi_i \in C^1([0, \infty))$ for all $i \ge 1$. Finally, the bound~\eqref{HMB} is an immediate consequence of~\eqref{TailIneq} with $r = 1$.
\end{proof}

\section{\textbf{Uniqueness of Classical Solutions}}\label{SEC4}
  
\begin{proof}[Proof of \Cref{TH3}]
To begin with, we observe that the non-negativity and monotonicity of $(\Lambda_{i}/i)_{i\ge 1}$ entail that
\begin{equation*}
	\frac{\Lambda_{i+1}^2}{i+1} = (i+1) \left( \frac{\Lambda_{i+1}}{i+1} \right)^2 \ge i \left( \frac{\Lambda_{i}}{i} \right)^2 = \frac{\Lambda_{i}^2}{i}, \qquad i\ge 1,
\end{equation*}	
so that $(\Lambda_{i}^2/i)_{i\ge 1}$ is also a non-negative and non-decreasing sequence. Moreover,	
$\Lambda_i^2 \ge \Lambda_1 \Lambda_i \ge \Lambda_i$ for $i \ge 1$ with $\Lambda_1^2 \ge 1$, so that~\eqref{ACond1} implies that $(\Gamma_{i,j})$ also satisfies~\eqref{ACond1} with $\Lambda_i^2$ instead of $\Lambda_i$. Taking into account~\eqref{UNIC}, we are in a position to apply \Cref{TH2} to deduce the existence of at least one mass-conserving global classical solution $\psi$ to~\eqref{FSNLBE} which satisfies the estimate~\eqref{UNME}. 

We now establish the uniqueness result stated in \Cref{TH3}. To this end, we follow the classical approach to uniqueness for coagulation-fragmentation equations, see \cite[Section~8.2.5]{BLL 2019} for instance and the references therein, and derive an estimate on the difference between two solutions in a suitably chosen weighted space. Such an approach is already used in \cite[Corollary~3.1]{AGP 24I} with a power weight and we adapt it here in the more general setting of \Cref{TH3}, the appropriate weight being the sequence $(\Lambda_i)_{i \ge 1}$. It is worth noting that, as usual, the proof requires slightly stronger conditions on the collision coefficients compared to those assumed for the existence result. 

Let us thus consider two classical solutions $\psi=(\psi_i)_{i\ge 1}$ and $\phi = (\phi_i)_{i\ge 1}$ to~\eqref{FSNLBE} satisfying~\eqref{UNME} and set $H := \psi- \phi$. Then, for $i \ge 1$, $H_i$ solves
\begin{align*}
	\frac{dH_i}{dt} = \sum_{j=i+1}^{\infty} \sum_{k=1}^{\infty} \Gamma_{j,k} \varphi_{i,j;k} [\psi_k H_j + \phi_j H_k] - (1-\delta_{i,1}) \sum_{k=1}^{\infty} \Gamma_{i,k} [\psi_k H_i + \phi_i H_k],
\end{align*}
from which we deduce that
\begin{align*}
	\frac{d}{dt}\sum_{i=1}^{\infty} \Lambda_i |H_i| &= \sum_{i=1}^{\infty} \sum_{j=i+1}^{\infty} \sum_{k=1}^{\infty} \Lambda_i \sgn(H_i) \varphi_{i,j;k} \Gamma_{j,k} \big[\psi_k H_j + \phi_j H_k\big] \\
	& \quad - \sum_{i=2}^{\infty} \sum_{k=1}^{\infty} \Lambda_i \sgn(H_i)\Gamma_{i,k} \big[\psi_k H_i + \phi_i H_k\big]\\
	& = \sum_{j=2}^{\infty} \sum_{k=1}^{\infty} \sum_{i=1}^{j-1} \Lambda_i \sgn(H_i) \varphi_{i,j;k} \Gamma_{j,k} \big[\psi_k H_j + \phi_j H_k\big] \\
	& \quad - \sum_{j=2}^{\infty} \sum_{k=1}^{\infty} \Lambda_j \sgn(H_j)\Gamma_{j,k} \big[\psi_k H_j + \phi_j H_k\big]\\
	& \le \sum_{j=2}^{\infty} \sum_{k=1}^{\infty} \sum_{i=1}^{j-1} \Lambda_i \varphi_{i,j;k} \Gamma_{j,k} \big[\psi_k |H_j| + \phi_j |H_k|\big] \\
	& \quad + \sum_{j=2}^{\infty} \sum_{k=1}^{\infty} \Lambda_j \Gamma_{j,k} \big[- \psi_k |H_j| + \phi_j |H_k|\big].
\end{align*}
Owing to~\eqref{LMC1} and the monotonicity of $(\Lambda_{i}/i)_{i\ge 1}$, we obtain
\begin{align*}
	\sum_{i=1}^{j-1} \Lambda_{i} \varphi_{i,j;k} \le \frac{\Lambda_j}{j} \sum_{i=1}^{j-1} i \varphi_{i,j;k} = \Lambda_j.
\end{align*}
Applying this bound gives
\begin{align*}
	 \frac{d}{dt}\sum_{i=1}^{\infty} \Lambda_i|H_i| & \le \sum_{j=2}^{\infty} \sum_{k=1}^{\infty} \Lambda_j \Gamma_{j,k} \big[\psi_k |H_j| + \phi_j |H_k|\big] \\
	 & \quad + \sum_{j=2}^{\infty} \sum_{k=1}^{\infty} \Lambda_j \Gamma_{j,k} \big[- \psi_k |H_j| + \phi_j |H_k|\big]\\
	 & = 2 \sum_{j=2}^{\infty} \sum_{k=1}^{\infty} \Lambda_j \Gamma_{j,k} \phi_j |H_k|
\end{align*}
Finally, we conclude from~\eqref{ACond1}, \eqref{UNME}, and the above inequality that
\begin{equation*}
	\frac{d}{dt} \sum_{i=1}^{\infty} \Lambda_i |H_i| \le 2 \sum_{j=2}^{\infty} \sum_{k=1}^{\infty} \Lambda_j^2 \Lambda_k \phi_j |H_k| 
\le 2 M_{\Lambda^2}(\psi^{\rm{in}}) \sum_{i=1}^{\infty} \Lambda_i |H_i|,
\end{equation*}
and an application of Gronwall's lemma completes the proof.
\end{proof}

\section{\textbf{Large Time Behavior}}\label{SEC5}

We now examine the asymptotic behavior of solutions in the long-time regime. As previously noted, the dynamics of this model ensure that clusters fragment solely into smaller components upon collision. This structural constraint suggests that, as $t \to \infty$, the system evolves toward a state where only $1$-clusters persist.

\begin{proposition}\label{Prop:LTB}
Suppose that $(\Gamma_{i,j})$ and $(\varphi_{i,j;k})$ satisfy~\eqref{ACond}, \eqref{LMC1}, and~\eqref{bCond}. For any $\psi^{\rm{in}} \in Y_{1,+}$ satisfying~\eqref{G0:Cond1} for some $G_0\in\mathcal{G}_{1,\infty}$, the mass-conserving global mild solution $\psi$ to~\eqref{FSNLBE} constructed in \Cref{TH1} satisfies
\begin{equation*}
	\lim_{t \to \infty} \|\psi(t) - \psi^{\infty}\|_1 = 0,
\end{equation*} 
for some limiting sequence $\psi^{\infty} = (\psi^{\infty}_i)_{i\ge 1} \in Y_{1,+}$. Furthermore, if $\Gamma_{i,i} > 0$ for some $i \geq 2$, then $\psi^{\infty}_i = 0$.
\end{proposition}

\begin{proof}
The proof proceeds along the same lines as that of \cite[Proposition~4.1]{PhL 2001}; see also \cite{ZHENG 2005}.
\end{proof}

\section{\textbf{Numerical Simulations}} \label{SEC6}

The final section is devoted to numerical simulations to illustrate the time evolution of the total number of clusters $\|\psi(t)\|_0$ of solutions to~\eqref{FSNLBE} and the large time behavior stated in \Cref{TH1} and \Cref{Prop:LTB}, respectively. For the numerical simulations of the nonlinear breakage system~\eqref{FSNLBE} provided below, we employ an implicit numerical scheme with a truncation parameter $p = 40$ in~\eqref{FTSNLBE}, and use the initial condition
\begin{equation}
	\psi_i^{\rm{in}} = \frac{1}{2^i}, \quad i \ge 1. \label{ICD}
\end{equation}
The simulations are performed for varying final times $T$, collision kernels $(\Gamma_{i,j})$, and daughter distribution functions $(\varphi_{i,j;k})$, using a fixed time step size $dt = 0.01$.

Among the chosen daughter distribution functions $(\varphi_{i,j;k})$ in the numerical simulations, we shall use that given by~\eqref{FDF:EX2}, which is unbounded. We first verify that it satisfies the condition~\eqref{bCond}.

\begin{lemma}\label{LEM:EX2}
	The fragment distribution function 
	\begin{equation*}
		\varphi_{i,j;k} = \frac{1}{2^i} \frac{j 2^{j-1}}{2^j - j - 1}, \qquad 1 \le i \le j-1, \quad j \ge 2, \quad k \ge 1, 
	\end{equation*}
	defined in~\eqref{FDF:EX2}, satisfies~\eqref{bCond} with $\alpha_0=\alpha_1=1$.
\end{lemma}

\begin{proof}
Introducing the function
\begin{equation*}
	\xi(z) := \frac{z 2^{z-1}}{2^z - z - 1}, \qquad z\ge 2,
\end{equation*}
we note that
\begin{equation*}
	\varphi_{i,j;k} = \frac{\xi(j)}{2^i}, \qquad 1 \le i \le i-1, \quad j\ge 2, \quad k\ge 1.
\end{equation*}
A tedious computation reveals that $\xi$ is increasing on $[4,\infty)$ with $\xi(2) = 4$, $\xi(3) = 3$ and $\xi(4) = 32/9$. As $\xi(3)<\xi(4)$, there holds
\begin{equation*}
	\varphi_{1,2;k} = \frac{\xi(2)}{2} \le \frac{1+\xi(3)}{2} \le 1 + \frac{\xi(k)}{2} = 1 + \varphi_{1,k;2}
\end{equation*}
for $k> j=2$ and
\begin{equation*}
	\varphi_{i,j;k} = \frac{\xi(j)}{2^i} \le \frac{\xi(k)}{2^i} \le 1 + \varphi_{i,k;j}, \qquad 1\le i \le j-1,
\end{equation*}
for $k>j\ge 3$. We complete the proof by noticing that~\eqref{bCond} with $\alpha_1=1$ is obviously satisfied for $k=j\ge 2$.
\end{proof}

Coming back to numerical simulations, we first mention that the total mass $\|\psi(t)\|_1$ is conserved during the computation, which is expected in view of~\eqref{TMC}.

\begin{figure}[htbp]
  \centering

  \begin{subfigure}[b]{0.48\textwidth}
    \centering
    \includegraphics[width=\linewidth]{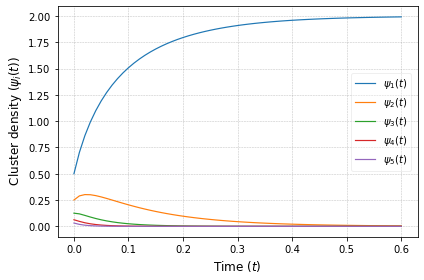}
    \caption{$\varphi_{i,j,k} =\frac{2}{j-1} $}
    \label{PSI_Case1}
  \end{subfigure}
  \hspace{0.02\textwidth}
  \begin{subfigure}[b]{0.48\textwidth}
    \centering
    \includegraphics[width=\linewidth]{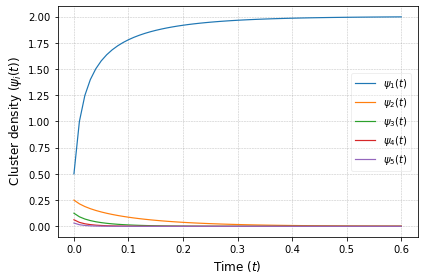}
    \caption{$\varphi_{i,j;k} = j \delta_{i,1}$}
    \label{PSI_Case2}
  \end{subfigure}

  \vspace{1em}

  \begin{subfigure}[b]{0.48\textwidth}
    \centering
    \includegraphics[width=\linewidth]{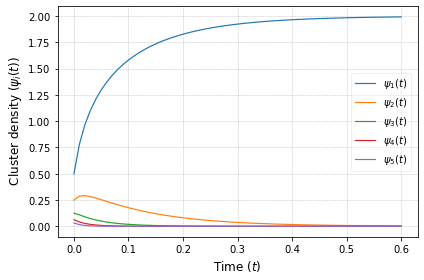}
    \caption{$ \varphi_{i,j;k} = \frac{1}{2^i} \frac{j 2^{j-1}}{2^j - j - 1}$}
    \label{PSI_Case3}
  \end{subfigure}

  \caption{Evolution of the cluster densities $\psi_i(t)$, $1\le i \le 5$, for fixed $(\varphi_{i,j;k})$ and  $\Gamma_{i,j} = i^2 j^2$}
  \label{PSI_Cases}
\end{figure}

We next turn to the evolution of cluster densities $\psi_i(t)$ under three different fragmentation mechanisms, while fixing the collision kernel $\Gamma_{i,j} = (ij)^2$. The outcome for $1\le i \le 5$ is depicted in \Cref{PSI_Cases}. In each case, the system evolves toward a monodisperse system, where all mass is ultimately distributed among $1$-clusters, a feature which is in agreement with the outcome of \Cref{Prop:LTB}. Indeed, \Cref{Prop:LTB} predicts that fragmentation drives the system toward $(\psi_1^\infty \delta_{i,1})_{i\ge 1}$, with $\psi_1^{\infty}=\|\psi^{\rm{in}}\|_1$, as $t \to \infty$. \Cref{PSI_Case1}, corresponding to uniform fragmentation, exhibits transient growth in small clusters before the mass concentrates in $1$-clusters. In \Cref{PSI_Case2}, where all collisions produce only $1$-clusters, the transition is faster, with minimal intermediate structure. \Cref{PSI_Case3}, which features an exponentially decaying daughter distribution, shows a similar trend to \Cref{PSI_Case1} but more strongly favours the formation of smaller fragments, thereby accelerating the depletion of larger clusters. These results underline how the specific form of the daughter distribution function shapes the dynamics and rate of convergence toward the monodisperse steady state.

\begin{figure}[htbp]
  \centering

  \begin{subfigure}[b]{0.48\textwidth}
    \centering
    \includegraphics[width=\linewidth]{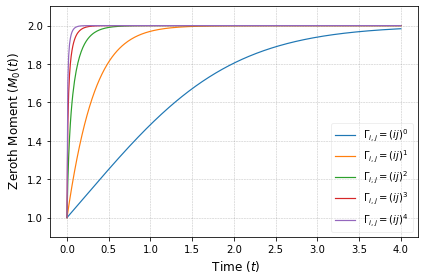}
    \caption{$\varphi_{i,j;k} =\frac{2}{j-1} $}
    \label{M0_Case1}
  \end{subfigure}
  \hspace{0.02\textwidth}
  \begin{subfigure}[b]{0.48\textwidth}
    \centering
    \includegraphics[width=\linewidth]{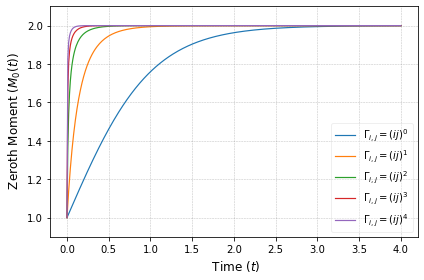}
    \caption{$\varphi_{i,j;k} = j \delta_{i,1}$}
    \label{M0_Case2}
  \end{subfigure}

  \vspace{1em}

  \begin{subfigure}[b]{0.48\textwidth}
    \centering
    \includegraphics[width=\linewidth]{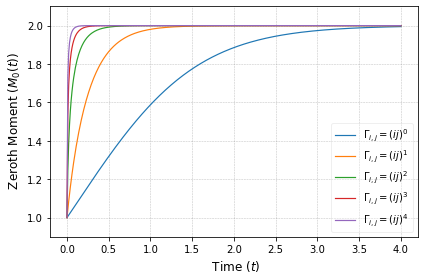}
    \caption{$ \varphi_{i,j;k} = \frac{1}{2^i} \frac{j 2^{j-1}}{2^j - j - 1}$}
    \label{M0_Case3}
  \end{subfigure}

  \caption{Evolution of the zeroth moment $\|\psi(t)\|_0$ for varying $(\Gamma_{i,j})$ and fixed $(\varphi_{i,j;k})$}
  \label{M0_Cases}
\end{figure}

We next compare the evolution of the zeroth moment $\|\psi(t)\|_0$, which quantifies the total number of clusters, under varying collision kernels $\Gamma_{i,j} = (ij)^{\alpha}$ for $0\le \alpha \le 4$ and fixed fragmentation laws $(\varphi_{i,j;k})$. The corresponding numerical simulations are displayed in \Cref{M0_Cases}. In all three cases, $\|\psi(t)\|_0$ increases from its initial value as larger clusters fragment into smaller ones, with the rate of increase depending sensitively on the exponent $\alpha$ in $(\Gamma_{i,j})$. Higher exponents enhance collision rates between large clusters, thereby hastening fragmentation and increasing $\|\psi(t)\|_0$ more rapidly. \Cref{M0_Case2}, with $\varphi_{i,j;k} = j \delta_{i,1}$, leads to the fastest saturation of $\|\psi(t)\|_0$ due to its exclusive production of monomers.  \Cref{M0_Case1} and \Cref{M0_Case3} exhibit similar trends, though the latter induces slightly faster dynamics due to its strong bias toward the production of smaller fragments. These results emphasize the interplay between collision intensity and daughter distribution in shaping the dynamics of cluster numbers.

\begin{figure}[htbp]
  \centering

  \begin{subfigure}[b]{0.48\textwidth}
    \centering
    \includegraphics[width=\linewidth]{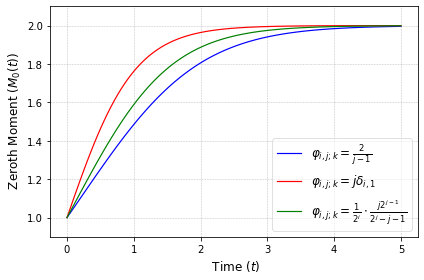}
    \caption{$\Gamma_{i,j} = 1$}
    \label{M0(Gamma=0)}
  \end{subfigure}
  \hspace{0.02\textwidth}
  \begin{subfigure}[b]{0.48\textwidth}
    \centering
    \includegraphics[width=\linewidth]{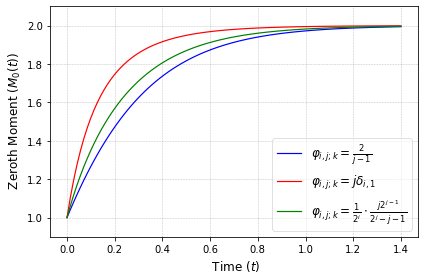}
    \caption{$\Gamma_{i,j} = ij$}
    \label{M0(Gamma=1)}
  \end{subfigure}

  \vspace{1em}

  \begin{subfigure}[b]{0.48\textwidth}
    \centering
    \includegraphics[width=\linewidth]{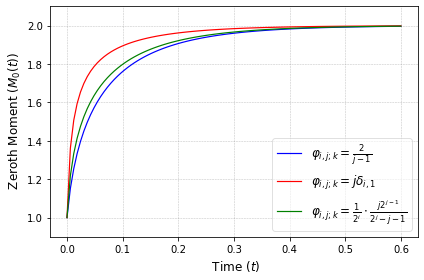}
    \caption{$\Gamma_{i,j} = i^2 j^2$}
    \label{M0(Gamma=2)}
  \end{subfigure}

  \caption{Evolution of the zeroth moment $\|\psi(t)\|_0$ for different  $(\Gamma_{i,j})$ and varying $\varphi_{i,j;k}$.}
  \label{M0_Gamma}
\end{figure}

We finally study in \Cref{M0_Gamma} the time evolution of the zeroth moment $\|\psi(t)\|_0$ for different daughter distribution functions $(\varphi_{i,j;k})$ under three collision kernels $(\Gamma_{i,j})$. The curves correspond to $\Gamma_{i,j} = 1$ (slowest rise), $\Gamma_{i,j} = i j$ (moderate rise), and $\Gamma_{i,j} = i^2 j^2$ (fastest rise), reflecting how collision rates scale with cluster sizes. The blue curve ($\varphi_{i,j;k} = \frac{2}{j-1}$) shows a moderate increase with uniform fragments, the red curve ($\varphi_{i,j;k} = j \delta_{i,1}$) rises fastest by producing only 1-clusters, and the green curve ($\varphi_{i,j;k} = \frac{1}{2^i} \frac{j 2^{j-1}}{2^j - j - 1}$) lies inbetween, favouring smaller fragments.

In conclusion, the dynamics of cluster fragmentation leads to a complete disintegration into {$1$-clusters which stops} further evolution, with the zeroth moment $\|\psi(t)\|_0$ stabilizing at the initial total mass $\|\psi^{\rm{in}}\|_1$. While the long term steady state features only $1$-clusters, the convergence to this state is distinctly influenced by both the fragment distribution function $(\varphi_{i,j;k})$ and the collision kernel $(\Gamma_{i,j})$, with higher kernel intensities accelerating the transition, as reflected in the different growth rates of $\|\psi(t)\|_0$ revealed by the numerical experiments.

\subsection*{Funding}  NA
\subsection*{Acknowledgements} MA expresses deep gratitude to Jindal Global Business School, O.P. Jindal Global University, for its invaluable support in providing essential resources.

\end{document}